\documentclass{amsart}

\usepackage{newcent}

\usepackage{mathrsfs}
\usepackage{amsthm}
\usepackage{amstext}
\usepackage{amssymb}
\usepackage{hyperref}
 
\theoremstyle{plain}
\newtheorem{theorem}[subsection]{Theorem}
\newtheorem{lemma}[subsection]{Lemma}

\newtheorem{proposition}[subsection]{Proposition}
\newtheorem{remark}[subsection]{Remark}
\newtheorem{example}[subsection]{Example}

\theoremstyle{definition}

\newtheorem{definition}[subsection]{Definition}

\newcommand{\R}{\mathbb{R}} 
\newcommand{\T}{\mathbb{T}} 
\newcommand{\Z}{\mathbb{Z}} 

\newcommand{\go}{\ensuremath{\Gamma^{0}}}
\newcommand{\HH}{\mathcal{H}}
\newcommand{\LL}{\mathcal{L}}
\newcommand{\TT}{\mathcal{T}}
\newcommand{\RR}{\mathcal{R}}
\newcommand{\Ind}{\operatorname{Ind}}
\newcommand{\cs}{$C^{*}$}
\newcommand{\ind}{\operatorname{ind}}
\newcommand{\Prim}{\operatorname{Prim}}
\newcommand{\Aut}{\operatorname{Aut}}
\newcommand{\Ad}{\operatorname{Ad}}
\newcommand{\sub}{\subset}
\newcommand{\hsg}{{\lambda}} 
\newcommand{\hsi}{{\nu}} 
\newcommand{\mfb}{\mathfrak{b}} 
\newcommand{\mfU}{\mathcal{U}} 

\usepackage[
backend=bibtex,
style=numeric,sortcites,
maxbibnames=99
]{biblatex}
\begin{document}

\title{Obstructions to lifting cocycles on groupoids and the associated $C^*$-algebras}

\author{Marius Ionescu}

\address{Department of Mathematics,  United States Naval Academy, 
  Annapolis, MD, USA}
 
\email{ionescu@usna.edu}
\thanks{The work of the authors was partially supported by their
  individual  grants from
  the Simons Foundation (\#209277 to Marius Ionescu and \# to Alex Kumjian).}

\author{Alex Kumjian}

\address{Department of Mathematics, University of Nevada, Reno, NV, USA}
\email{alex@unr.edu}

\thanks{This work was supported by Simons Foundation Collaboration grants {\#}209277 (MI) and  {\#}353626 (AK)}
\begin{abstract}
    Given a short exact sequence of locally compact abelian groups $0 \to A \to B \to C \to 0$
     and a continuous $C$-valued $1$-cocycle $\phi$ on a
    locally compact Hausdorff groupoid $\Gamma$ we construct a twist
    of $\Gamma$ by $A$ that is trivial if and only if $\phi$
    lifts. The cocycle determines a strongly continuous action of
    $\widehat{C}$ into $\Aut C^*(\Gamma)$ and we prove that the
    $C^*$-algebra of the twist is isomorphic to the induced algebra
    of this action if $\Gamma$ is amenable. 
    We apply our results to a groupoid determined by a
    locally finite cover of a space $X$ and a cocycle provided by a
    \v{C}ech 1-cocycle with coefficients in the sheaf of germs of
    continuous $\mathbb{T}$-valued functions. We prove that the
    $C^*$-algebra of the resulting twist is continuous trace and we compute
    its Dixmier-Douady invariant.
  \end{abstract}

  \maketitle

\section{Introduction}
\label{sec:introduction}
It is a well-known fact that given a $1$-cocycle $\phi : \Gamma \to \T$ on 
an \'{e}tale groupoid $\Gamma$, there is an automorphism $\alpha$ of $C^*(\Gamma)$
such that $\alpha(f)(\gamma) = \phi(\gamma)f(\gamma)$ for all $f \in C_c(\Gamma)$
(see \cite[Proposition II.5.1]{ren:groupoid}).
If $\phi$ can be lifted to an $\R$-valued $1$-cocycle $\tilde{\phi}$, there is a strongly continuous 
$1$-parameter group of automorphisms $\tilde{\alpha}: \R \to \Aut C^*(\Gamma)$ such that
$\alpha = \tilde{\alpha}_1$.  
But in general there is a cohomological obstruction to lifting $\phi$.   
There is a central groupoid extension $\Sigma_\phi$ of $\Gamma$ by $\Z$, called a twist, 
which is trivial precisely when $\phi$ can be lifted.  
Our  goal in this paper is to describe the structure of $C^*(\Sigma_\phi)$ in terms of
$C^*(\Gamma)$ and the automorphism $\alpha$.  

Our results will be proven in a somewhat more general form.  
Given a locally compact Hausdorff groupoid $\Gamma$, a short exact
sequence of locally compact abelian groups 
\[
  0\rightarrow A\xrightarrow{i} B\xrightarrow{p} C\rightarrow 0
  \]
and a continuous $1$-cocycle $\phi : \Gamma \to C$ we construct a twist $\Sigma_\phi$ of $\Gamma$ by $A$
which is trivial if and only if the cocycle lifts.  
As above the cocycle $\phi$ determines  a strongly continuous action
$\alpha^{\phi} :\widehat{C}\to \Aut C^*(\Gamma)$.  
We prove that $C^*(\Sigma_\phi)$ is isomorphic to the induced algebra for this action
if $\Gamma$ is amenable 
(see Theorem~\ref{thm:main}). 
So in particular, there is an action $\gamma: \widehat{B}\to \Aut C^*(\Sigma_\phi)$
such that $C^*(\Sigma_\phi) \rtimes_\gamma \widehat{B}$ is strong Morita equivalent to
$C^*(\Gamma) \rtimes_{\alpha^{\phi}}\widehat{C}$.

In example \ref{ex:cech} we consider a groupoid $\Gamma$ determined by a locally finite open cover 
$\mfU := \{ U_i \}_{i \in I}$ of a space $X$ (see \cite{RaeTay85}) and
take $A := \Z$, $B := \R$ and $C := \T$. 
The $1$-cocycle $\phi$ arises from a \v{C}ech 1-cocycle $\lambda = \{ \lambda_{ij} \}_{i, j \in I}$
with coefficients in $\TT$, the sheaf of germs of continuous $\T$-valued functions.
If $\lambda$ satisfies a certain liftability hypothesis, the corresponding twist $\Sigma_\phi$ 
is then determined by a \v{C}ech 2-cocycle $\lambda^\star$ with coefficients in the constant sheaf 
with fiber $\Z$ (we denote the sheaf by $\Z$ for the sake of simplicity) 
that measures the obstruction to lifting $\lambda$ to a \v{C}ech 1-cocycle with coefficients 
in the sheaf of germs of continuous $\R$-valued functions.
Finally we show in example \ref{ex:cech2} that $C^*(\Sigma_\phi)$ is a continuous trace algebra 
and compute its Dixmier-Douady invariant using the cohomology class $[\lambda]$
which is identified with $[\lambda^\star]$ via the standard isomorphism of \v{C}ech 
cohomology groups $\check{H}^1(X, \TT) \cong \check{H}^2(X, \Z)$. This
example was a key motivation for this project and it was suggested by
results of \cite{RaeRos88} (see also \cite{wil:book07}).
 
In this note we assume that all topological spaces and groupoids are second countable 
locally compact and Hausdorff and all groups are second countable abelian locally compact 
and Hausdorff.  Hence all spaces are paracompact.
Moreover we assume that all groupoids are endowed with a Haar
system.

The second author would like to thank the first author and his colleagues at the 
Naval Academy for their hospitality and support during a recent visit.

\section{Preliminaries}
\label{sec:preliminaries}

In this section we present first a characterization of the induced
algebra from an action of a closed subgroup of a locally compact abelian group on a $C^*$-algebra. 
Our characterization is essential in our proof of the main result, Theorem
\ref{thm:main}. For a related characterization of the induced algebra
see \cite{Ech90}. Second, we provide conditions that guarantee that
the natural map $j$ as defined in equation \eqref{eq:embedding} from the
$C^*$-algebra of a subgroupoid into the 
multiplier algebra of the $C^*$-algebra of the groupoid is
faithful. Our results generalize well known facts about group
$C^*$-algebras. 

\subsection{A characterization of the induced algebra}
\label{subsec:char}
Let $H$ be a closed subgroup of a locally compact abelian group $G$, let $D$ be
a \cs-algebra 
and let $\alpha : H \to \Aut D$ be a strongly continuous action.  
If $f: G \to D$ is a continuous function such that $f(x-h) = \alpha_h(f(x))$ 
for all $h \in H, x \in G$ then $x+H \mapsto \| f(x) \|$ yields a well-defined 
continuous function.
We define the \emph{induced \cs-algebra} (see \cite[\S3.6]{wil:book07}) by
\begin{gather*}
\ind_H^G(D, \alpha) := \{ f: G \to D \text{ continuous }\mid f(x-h) = \alpha_h(f(x)), h \in H, x \in G \\
\text{ and } x + H \mapsto \| f(x) \| \in C_0(H/G) \}.
\end{gather*}
There is a canonical translation action $\beta: G \to \Aut(\ind_H^G(D, \alpha))$ 
given by $\beta_g(f)(x) := f(x-g)$.  
By \cite[Lemma 3.1]{RaeRos88} $D \rtimes_\alpha H$ is strong Morita equivalent to
$\ind_H^G(D, \alpha) \rtimes_\beta G$.

Point evaluation at $0$ yields a surjective homomorphism 
$\pi : \ind_H^G(D, \alpha) \to D$ given by $\pi(f) = f(0)$. 
There is a natural translation action 
\[
  \tau : G \to  \Aut (C_0(G/H))
\]
(where $\tau_g(k)(x + H) = k(x - g + H)$) and a $G$-equivariant homomorphism 
\[
  j: C_0(G/H) \to M(\ind_H^G(D, \alpha))
\]
determined by pointwise multiplication.
It is easy to check that these maps satisfy the following conditions 
for all $f \in \ind_H^G(A, \alpha)$:
\begin{itemize}
\item[i.] For all $k \in C_0(G/H)$,
$\pi(j(k)f) = k(H)\pi(f)$;
\item[ii.] 
for all $h \in H$, $\pi(\beta_h(f)) = \alpha_h(\pi(f))$;
\item[iii.] 
if $\pi(\beta_g(f)) = 0$ for all $g \in G$, then $f = 0$;
\item[iv.] 
and
\[
\lim_{x + H \to \infty} \| f(x) \| = 0.
\]
\end{itemize}

We show below that these conditions characterize the induced algebra. 
Note that $\ind_H^G(D, \alpha)$ may be identified with the section algebra of a continuous 
$C^*$-bundle over $G/H$ with fibers isomorphic to $D$ and factor maps $\pi_{x + H}$.

\begin{theorem}\label{thm:indalg}
Let $H$ be a closed subgroup of a locally compact abelian group $G$. Let $D$
and $E$ be \cs-algebras, 
$\alpha : H \to \Aut D$ and  $\gamma : G \to \Aut E$  strongly continuous
actions,  
$\rho : E \to D$ a surjective homomorphism, and 
 $i : C_0(G/H) \to Z(M(E))$ a $G$-equivariant homomorphism.
Suppose that 
\begin{itemize}
\item[i] 
For all $k \in C_0(G/H)$ and $e \in E$, $\rho(i(k)e) = k(H)\rho(e)$;
\item[ii] 
For all $h \in H$ and $e \in E$, $\rho(\gamma_h(e)) = \alpha_h(\rho(e))$;
\item[iii] 
If $\rho(\gamma_g(e)) = 0$ for all $g \in G$, then $e = 0$;
\item[iv] 
For all $e \in E$, 
\[
\lim_{x + H \to \infty} \| \rho(\gamma_x(e)) \| = 0.
\]
\end{itemize}
Then there is a (unique) $G$-equivariant isomorphism $\Psi : E \to \ind_H^G(D, \alpha)$
such that: $\pi\circ\Psi = \rho$ and $j(k)\Psi(e) = \Psi(i(k)e)$ 
for all $k \in C_0(G/H)$ and $e \in E$.
\end{theorem}
\begin{proof}
For $e \in E$, we define a function $\Psi(e) : G \to D$ by $\Psi(e)(x) := \rho(\gamma_{-x}(e))$.
It is straightforward to check that $\Psi(e)$ is continuous and that 
$\Psi(e) \in \ind_H^G(D, \alpha)$.   Indeed we have
\[
\Psi(e)(x - h) = \rho (\gamma_{h-x}(e)) = \rho (\gamma_h(\gamma_{-x}(e)))  = \alpha_h(\rho(\gamma_{-x}(e)))
= \alpha_h(\Psi(e)(x)).
\]
It is also routine to check that the map thus defined $\Psi: E \to \ind_H^G(D, \alpha)$ 
is an equivariant $*$-homormorphism.  Injectivity follows from (iii) above.  
Let  $k \in C_0(G/H)$ and $e \in E$.  Then by (i) above and the $G$-equivariance of $i$
it follows that for all $x \in G$
\begin{align*}
\Psi(i(k)e)(x) &= \rho(\gamma_{-x}(i(k)e)) = \rho(i(\tau_{-x}(k))\gamma_{-x}(e)) 
               = \tau_{-x}(k)(H)\rho(\gamma_{-x}(e)) \\
               &= k(x+H)\Psi(e)(x) = (j(k)\Psi(e))(x),
\end{align*}
and hence $j(k)\Psi(e) = \Psi(i(k)e)$.  Thus  $\Psi$ is a map of $C_0(G/H)$-algebras.  
Since $\rho(e) = \Psi(e)(0)$ for all $e \in E$ we have $\pi\circ\Psi = \rho$.

Next we use Proposition C.24 of  \cite{wil:book07} (see also \cite[Proposition 14.1]{feldor:vol1}),
which states that a submodule of continuous sections of an upper semicontinuous $C^*$-bundle which is 
fiberwise dense must also be norm dense in the $C^*$-algebra of sections vanishing at infinity,
to prove that $\Psi(E)$ is dense in $\ind_H^G(D, \alpha)$.  
This will prove that  $\Psi$ is surjective and therefore an isomorphism.
We check the hypotheses a) and b) of the proposition.
Let $f \in \Psi(E)$ and let $k \in C_0(G/H)$, then $f = \Psi(e)$ for some $e \in E$ and so
\[
j(k)f = j(k)\Psi(e) = \Psi(i(k)e) \in \Psi(E).
\]
This proves condition (a).  
Now since each factor map is conjugate to $\pi\circ\beta_x$ for some $x \in G$, 
its restriction to $\Psi(E)$ is surjective.
Thus (b) holds and the result follows.
\end{proof}

\begin{remark}\label{rem:nondeg}
Let $D_1$ and $D_2$ be $C^*$-algebras.  A homomorphism $\phi: D_1 \to M(D_2)$ is said to be \emph{nondegenerate} 
if it extends uniquely to a unital map $\phi: M(D_1) \to M(D_2)$ or equivalently 
if $\{ \phi(e_\lambda) \}_\lambda$ converges strictly to the unit of $M(D_2)$
for every approximate identity $\{ e_\lambda \}_\lambda$ in $D_1$.
Property (iv) above is equivalent to the requirement that $i: C_0(G/H) \to M(E)$  be nondegenerate.
\end{remark}

\begin{remark}\label{rem:dual}
  Note that since $C_0(G/H) \cong C^*(\widehat{G/H}))$, the map $i$ is determined by a strictly continuous homomorphism 
$u : \widehat{G/H} \to UM(E)$  where $UM(E)$ is the unitary group of the multiplier algebra $M(E)$.  
The nondegeneracy of $i$ is equivalent to the requirement that $u_0 = 1$ and
condition (i) above is satisfied if and only if
$\rho(u_{\chi}e) = \rho(e)$ for all $\chi \in \widehat{G/H}$ and $e \in E$.
\end{remark}

\begin{remark}
Our characterization of the induced algebra is provided to facilitate the proof of our main result, namely, that the $C^*$-algebra of a certain groupoid extension regarded as an obstruction to lifting a cocycle is an induced $C^*$-algebra.  
One could in principle use Echterhoff's more general characterization (see the main theorem of  \cite{Ech90}), but this would have required the identification of the $C^*$-algebra of the quotient groupoid  with a certain quotient of the putative induced algebra.
The proof of this identification would use similar techniques with those in our proof and it would not necessarily lead to a simplification of our arguments. 

\end{remark}

\subsection{Multiplier algebras of groupoid $C^*$-algebras}
\label{subsec:mult-algebr-group}
Given a $C^*$-algebra $D$ let $M(D)$ denote its multiplier algebra. We
view $M(D)$ as the $C^*$-algebra $\LL(D_D)$  of adjointable maps
on the Hilbert $C^*$-module $D_D$ (see, for example, \cite[Section
2.3]{rw:morita} and \cite{lan:hilbert}). Recall that $D_D$ is a full Hilbert $C^*$-module
via $d\cdot e=de$ and $\langle d,e\rangle_D=d^*e$.

Assume that $\Sigma$ is a closed subgroupoid of a locally
compact Hausdorff groupoid $\Gamma$  such that $\Sigma^0=\go$. Assume
that $\Sigma$ is endowed with a Haar system
$\beta=\{\beta^u\}_{u\in\go}$ and that $\Gamma$ is endowed with a Haar
system $\hsg=\{\lambda^u\}_{u\in\go}$.
Then  there is a $*$-homomorphism $j : C^*(\Sigma) \to M(C^*(\Gamma))$
defined for $a\in C_c(\Sigma)$ and $f\in C_c(\Gamma)$ via
\begin{equation}\label{eq:embedding}
(j(a)(f))(\gamma)=\int_{\Sigma}a(\eta)f(\eta^{-1}\gamma)\,d\,\beta^{r(\gamma)}(\eta)
\end{equation}
such that $j(a)^*=j(a^*)$ (see \cite[Proposition II.2.4]{ren:groupoid}).
It is known that if $\Sigma$ and $\Gamma$ are locally compact
\emph{groups} then the map $j$
fails to be faithful in general (see, for example,
\cite{BekVal_Ast95}). However,  if $\Sigma$ is a clopen subgroup of $\Gamma$
 or $\Gamma$ is an amenable group then the map $j$ is a faithful
*-homomorphism of $C^{*}(\Sigma)$ into $C^{*}(\Gamma)$
(see \cite[Proposition 1.2]{rie:aim74}, \cite[Theorem 1.3]{BekVal_Ast95}, \cite[Corollary
1.5]{BekLauSch_MatAnn92}). 

We prove next that the above
mentioned results hold for groupoid $C^{*}$-algebras as well: the map
$j$ is faithful if $\Sigma$ is a clopen subgroupoid of $\Gamma$ or if
$\Gamma$  is an amenable groupoid. 

Before proceeding further, we recall the definition of induced
representations from closed subgroupoids following
\cite[Section 2]{ionwil:pams08} (see also \cite[Section
II.2]{ren:groupoid}). Let $\Sigma^{\Gamma}:=\Gamma*\Gamma/\Sigma$ be 
the imprimitivity groupoid. Then $C_{c}(\Gamma)$ is a
pre-$C_{c}(\Sigma^{\Gamma})-C_{c}(\Sigma)$-imprimitivity bimodule
with actions and inner products given by

\begin{align*}  
  F\cdot \varphi(z)&=\int_{\Gamma} F\bigl([z,y]\bigr) \varphi(y)\,d\hsg^{s(z)}
  (y) \\
\varphi\cdot g(z)&=\int_{\Sigma} \varphi(z h)g(h^{-1}) \,d\beta^{s(z)} (h)
\\
\langle\varphi,\psi\rangle_{\ast}(h)&= \int_{\Gamma} \overline{\varphi(y)} \psi(y h)\,
d\hsg^{r(h)} (y) \\
_{\ast}\langle\varphi,\psi\rangle\bigl([x,y]\bigr) &= \int_{\Sigma} \varphi(x
h)\overline{\psi(y h)} \, d\beta^{s(x)}(h).
\end{align*}
If $L$ is a representation of $C^{*}(\Sigma)$ on $B(\HH_{L})$ then the induced representation
$\Ind_{\Sigma}^{\Gamma}L$ acts on the completion $\HH_{\Ind L}$ of
$C_{c}(\Gamma)\odot \HH_{L}$ with respect to the pre-inner product
given on elementary tensors by
  \begin{equation}\label{eq:innprod}
    (\varphi\otimes h\,,\,\psi\otimes k)=(L(\langle \psi,\varphi\rangle_{\ast})h\,,\,k).
  \end{equation}
If $\varphi\otimes_{\Sigma}h$ denotes the class of $\varphi\otimes h$ in
$\HH_{\Ind L}$, then the induced representation is given by
\begin{equation}\label{eq:indrep}  
  \Ind_{\Sigma}^{\Gamma}L(f)(\varphi\otimes_{\Sigma}h)=f*\varphi\otimes_{\Sigma}h
\end{equation}
for $f\in C_{c}(\Gamma)$, where
  \[
   f*\varphi(\gamma)=\int_{\Gamma}
   f(\eta)\varphi(\eta^{-1}\gamma)\,d\hsg^{r(\gamma)}(\eta).     
   \]
In the following we suppress $\Sigma$ from $\varphi\otimes_{\Sigma}h$
to simplify slightly the notation.

\subsubsection{The clopen case}

 Assume first that $\Sigma$ is a clopen
subgroupoid of $\Gamma$. Then the
restriction of the Haar system  $\hsg=\{\hsg^{u}\}_{u\in\go}$ to $\Sigma$ is a Haar
system on $\Sigma$. We assume in the following that $\Sigma$ is
endowed with this Haar system, that is, $\beta=\lambda|_{\Sigma}$. 
Then the map $i_{\Sigma}:C_{c}(\Sigma)\to C_{c}(\Gamma)$ defined via
\begin{equation*}
  i_{\Sigma}(f)(\gamma)=
  \begin{cases}
    f(\gamma)&\text{ if }\gamma\in \Sigma\\
    0&\text{otherwise}
  \end{cases}
\end{equation*}
is a well defined faithful $*$-homomorphism. We prove next that
$i_{\Sigma}$ extends to a faithful $*$-homomorphism of
$C^{*}(\Sigma)$ into $C^{*}(\Gamma)$, generalizing Proposition 1.2 of
\cite{rie:aim74}. 

\begin{proposition}\label{prop:discretecase}
  With notation as above, $i_{\Sigma}$
  extends to an embedding $i_{\Sigma}:C^{*}(\Sigma)\to C^{*}(\Gamma)$
  and, therefore, $C^{*}(\Sigma)$ can be viewed as a subalgebra of $C^{*}(\Gamma)$.
\end{proposition}
\begin{proof}
  We need to show that $\Vert i_{\Sigma}(f)\Vert_{C^{*}(\Gamma)}=\Vert
  f\Vert_{C^{*}(\Sigma)}$ for all $f\in C_{c}(\Sigma)$. Let $L$ be a
  representation of $C^{*}(\Gamma)$. Renault's disintegration theorem
  (see \cite[Proposition 4.2]{ren:jot87} and also \cite[Section
  7]{muhwil:nyjm08}) implies that $L$ is the integrated form of a
  unitary representation $(\mu,\go *\HH,V)$  of $\Gamma$. Since $\Sigma$
  is a clopen subgroupoid of $\Gamma$, any unitary representation
  of $\Gamma$ restricts to a unitary representation of
  $\Sigma$. Therefore $\Vert f\Vert_{C^{*}(\Sigma)}\ge \Vert
  i_{\Sigma}(f)\Vert_{C^{*}(\Gamma)}$ for all $f\in C_{c}(H)$.

  For the converse inequality, let $(L,\HH_{L})$ be a representation of
  $C^{*}(\Sigma)$. Since $\Sigma$ is a closed subgroupoid of
  $\Gamma$, the representation $L$ can be induced to a representation
  $\Ind_{\Sigma}^{\Gamma}L$ of $C^{*}(\Gamma)$
  as in \eqref{eq:indrep}.   
   Let $\HH_{\text{res}}$ be the closed subspace of $\HH_{\Ind\,L}$
   obtained by completing $C_c(\Sigma)\odot \HH_L$ via the inner product
   in \eqref{eq:innprod}. Then $U:\HH_{\text{res}}\to \HH_L$ defined
   on elementary tensors    via $U(\varphi\otimes h)=L(\varphi)h$ defines
   a unitary that intertwines $L$ with a subrepresentation of
   $\Ind_{\Sigma}^{\Gamma}L$ restricted to $C^*(\Sigma)$. It follows that
   $\Vert f\Vert_{C^{*}(\Sigma)}\le \Vert 
  i_{\Sigma}(f)\Vert_{C^{*}(\Gamma)}$ for all $f\in
  C_{c}(\Sigma)$. Therefore $\Vert i_{\Sigma}(f)\Vert_{C^{*}(\Gamma)}=\Vert
  f\Vert_{C^{*}(\Sigma)}$ for all $f\in   C_{c}(\Sigma)$ and one can view
  $C^*(\Sigma)$ as a subalgebra of $C^*(\Gamma)$.
\end{proof}

\subsubsection{The amenable case}
\label{sec:amenable-case}

We assume now that $(\Sigma,\beta)$ is a closed subgroupoid of $(\Gamma,
\hsg)$ and   $\Gamma$ is amenable in the
sense of \cite{anaren:amenable00}. It follows that $\Sigma$ is
amenable as well (see \cite[Proposition 5.1.1]{anaren:amenable00}).
\begin{proposition}\label{prop:amenablecase}
Assume that  $(\Sigma,\beta)$ is a closed subgroupoid of an amenable groupoid $(\Gamma,
\hsg)$ such that $\Sigma^0=\go$. Then the map $j$ defined in
\eqref{eq:embedding} is faithful.
\end{proposition}
 We recall first the definition of the left regular representation of
 $C^*(\Gamma)$. Let $\mu$ be a quasi-invariant measure
on $\go$ with full support. The \emph{left regular representation} of $C^*(\Gamma)$ is the
representation $L_\Gamma:=\Ind_{\go}^\Gamma\mu$ induced from
$\mu$. By using induction in stages (see \cite[Theorem 4]{ionwil:pams08}) it
follows that if $\Sigma$ is a closed subgroupoid of $\Gamma$ such that
$\Sigma^0=\go$ then
$L_\Gamma$ is unitarily equivalent to $\Ind_\Sigma^\Gamma L_\Sigma$. 

\begin{proof}[Proof of Proposition \ref{prop:amenablecase}]
  The proof is virtually
  identical with the group case (see, for example, the proof of
  \cite[Thm. 1.3]{BekVal_Ast95}): by the amenability of $\Sigma$, any
  representation of $C^*(\Sigma)$ is weakly contained in the left regular
  representation of $C^*(\Sigma)$, which is itself  contained in the restriction to
  $C^*(\Sigma)$ of the left regular representation of $C^*(\Gamma)$.
\end{proof}
The authors would like to thank Alcides Buss and Dana P. Williams for
useful conversations and suggestions that led to a significant
shortening of our
original 
proof of Proposition \ref{prop:amenablecase}.

\section{Twists and Short Exact Sequences}
\label{sec:Twists}

Let $\Gamma$ be a locally compact Hausdorff groupoid and let $G$ be a
locally compact Hausdorff abelian group. 
The set of continuous 1-cocyles from $\Gamma$ to $G$ is defined via
\[
Z_\Gamma(G)=Z^1(\Gamma,G):=\{\,\phi:\Gamma\to
G\,|\,\phi(\gamma_1\gamma_2)=\phi(\gamma_1)+\phi(\gamma_2)\,\text{for
  all}\,(\gamma_1,\gamma_2)\in \Gamma^2\,\}.
\]
Then $Z_\Gamma(G)$ is an abelian group and the map $G\mapsto
Z_\Gamma(G)$ is a functor.
 
\begin{definition}
  Let $A$ be an abelian group and $\Gamma$ a groupoid. A \emph{twist} by
  $A$ over $\Gamma$ is a central groupoid extension 
  \begin{equation*}
    \go\times A\xrightarrow{j}\Sigma\xrightarrow{\pi} \Gamma,
  \end{equation*}
  where $\Sigma^{0}=\go$, $j$ is injective, $\pi$ is surjective, and
  $j(r(\sigma),a)\sigma=\sigma j(s(\sigma),a)$ for all $\sigma\in
  \Sigma$ and $a\in A$.
\end{definition}

\begin{example}
  The \emph{semi-direct} product $\Gamma\times A$ of $\Gamma$ and $A$
   is called the \emph{trivial twist}. Recall from \cite{kum:jot88} that 
  $(\gamma_{1},a_{1})(\gamma_{2},a_{2})=
  (\gamma_{1}\gamma_{2},a_{1}+a_{2})$ provided that
  $s(\gamma_{1})=r(\gamma_{2})$, and
  $(\gamma,a)^{-1}=(\gamma^{-1},-a)$. Then $\Gamma\times A$ is a twist
  by $A$ via $j_0(u,a)=(u,a)$ for $(u,a)\in \go\times A$, and
  $\pi_0(\gamma,a)=\gamma$ for $(\gamma,a)\in \Gamma\times A$.
\end{example}
Following Definition 2.5 of \cite{kum:jot88} we say that two twists by
$A$ are  properly isomorphic if there is a twist morphism
between them which preserves the inclusion of $\go\times A$.
The following lemma is a generalization of Proposition 2.2 of \cite{kum:jot88}.
\begin{lemma}\label{lem:trivial}
  A twist $\Sigma$ by $A$ is properly isomorphic to a trivial twist if and only if there is a
  groupoid homomorphism $\tau:\Gamma\to\Sigma$ such that $\pi\tau=\operatorname{id}_\Gamma$.
\end{lemma}
\begin{proof}
  If $\tilde{\tau}:\Gamma\times A\to \Sigma$ is a twist isomorphism then we can define
  $\tau:\Gamma\to \Sigma$ via $\tau(\gamma)=\tilde{\tau}(\gamma,0_A)$, where $0_A$
  is the identity of $A$. It is easy to check that $\tau$ is a
  groupoid homomorphism and that $\pi\tau=\operatorname{id}_\Gamma$.

  Assume now that there is a groupoid homomorphism $\tau:\Gamma\to
  \Sigma$  such that $\pi\tau=\operatorname{id}_\Gamma$. Then we can define 
  $\tilde{\tau}:\Gamma\times A\to \Sigma$ via
  $\tilde{\tau}(\gamma,a)=\tau(\gamma)j(s(\gamma),a)$. We check next
  that $\tilde{\tau}$ is a twist isomorphism. Let $(u,a)\in\go\times
  A$. Then
  \[
  \tilde{\tau}(j_0(u,a))=\tilde{\tau}(u,a)=\tau(u)j(u,a)=\tau(u)\tau(u)j(u,a)=\tau(u)j(u,a)\tau(u)=j(u,a). 
  \]
  If $(\gamma,a)\in \Gamma\times A$ then
  \[
  \pi(\tilde{\tau}(\gamma,a))=\pi(\tau(\gamma)j(s(\gamma),a))=\gamma\pi(j(s(\gamma),a))=\gamma=\pi_0(\gamma,a). 
  \qedhere \]
\end{proof}
Following \cite{kum:jot88} we write $T_{\Gamma}(A)$ for the collection
of proper isomorphism classes of twists by $A$ and we write
$[\Sigma]\in T_{\Gamma}(A)$. We endow $T_{\Gamma}(A)$ with the
operation
$[\Sigma]+[\Sigma^{\prime}]:=[\nabla_{*}^{A}(\Sigma*_{\Gamma}\Sigma^{\prime})]$,
where $\nabla^{A}\in \operatorname{Hom}_{\Gamma}(A\oplus A,A)$ is
defined via $\nabla^{A}(a,a^{\prime})=a+a^{\prime}$ (see
\cite[Proposition 2.6]{kum:jot88}). Then $T_{\Gamma}(A)$ is an abelian group
with neutral element $[\Gamma\times A]$.  
It can be shown that $A\mapsto T_\Gamma(A)$ is a half-exact functor.

Suppose that $B$ and $C$ are locally compact  abelian groups. If
$p:B\to C$ is a homomorphism and $\phi\in Z_{\Gamma}(C)$,
then the \emph{obstruction twist} determined by $\phi$  is defined via
\begin{equation}
  \label{eq:pullback}
  \Sigma_{\phi}=\{(\gamma,b)\in \Gamma\times B\,:\,\phi(\gamma)=p(b)\}.
\end{equation}
We establish that $\Sigma_{\phi}$ is indeed a twist in the following
proposition and show that it has a Haar system in the next section (see
equation \eqref{eq:HaarSyst} in Section \ref{sec:struct}).

If $\Gamma$ is  an \'etale groupoid, it has a basis consisting of open bisections, that is,
open subsets to which the restrictions of both the range and source maps are injective.  

\begin{proposition}\label{lem:twist_prop}
Assume that $p:B\to C$ is a surjective homomorphism of locally compact abelian
groups and let $\phi\in Z_{\Gamma}(C)$. Then $\Sigma_{\phi}$ is a
closed subgroupoid of $\Gamma\times B$ and it is a twist of $\Gamma$
by $A:=\ker p$. $\Sigma_{\phi}$ is properly isometric to the trivial twist if and only if there is
$\tilde{\phi}\in Z_{\Gamma}(B)$ such that $\phi=p_{*}\tilde{\phi}$. The projection
$\pi_1:\Sigma_{\phi}\to \Gamma$ is a surjective groupoid
morphism. If $\Gamma$ is 
\'etale and $A$ is discrete then $\Sigma_{\phi}$  is \'etale 
 and $\pi_1$ is a local homeomorphism. 
\end{proposition}
\begin{proof}
  Note that $\Sigma_{\phi}$ is a closed subgroupoid of $\Gamma\times
  B$ because $\phi$ and $p$ are continuous maps. Since $p$
  is surjective it follows that 
  $\Sigma_{\phi}$ is a twist by $A$ via $j(u,a)=(u,a)$ and
  $\pi(\gamma,b)=\gamma$. We make the obvious identification
  $\Sigma_{\phi}^0 = \go\times \{0_B\}\simeq \go$, where $0_B$ is the unit of $B$. If
  $\Sigma_{\phi}$ is properly isometric to the trivial twist then Lemma 
  \ref{lem:trivial} implies that there is a groupoid homomorphism
  $\tau:\Gamma\to \Sigma_{\phi}$ such that
  $\pi_1\tau=\operatorname{id}_{\Gamma}$. Therefore there is
  $\tilde{\phi}\in Z_{\Gamma}(B)$ such that
  $\tau(\gamma)=(\gamma,\tilde{\phi}(\gamma))$ and, hence,
  $\phi=p_{*}\tilde\phi$. Conversely, if $\phi=p_*\tilde{\phi}$ for
  some $\tilde{\phi}\in Z_\Gamma(B)$, then define $\tau:\Gamma\to
  \Sigma_\phi$ via $\tau(\gamma)=(\gamma,\tilde{\phi}(\gamma))$. It
  follows that $\pi_1\tau=\operatorname{id}_{\Gamma}$ and Lemma
  \ref{lem:trivial} implied that $\Sigma_\phi$ is properly isometric to the trivial twist.
    
  It is easy to check that $\pi_1$ is a groupoid morphism.  
  Moreover, $\pi_1$ is surjective since $p$ is surjective.

  Now suppose that $\Gamma$ is \'etale and $A$ is discrete. 
  Since $\pi_1: \Sigma_{\phi}\to \Gamma$ is open, in order to prove that $\pi_1$ 
  is a local homeomorphism it suffices to show that it is locally injective.
  Let $(\gamma, b)\in \Sigma_{\phi}$. 
  Then there are an open bisection $U$ in $\Gamma$ such that $\gamma \in U$
  and a neighborhood $V_0$ of $0_B$ such that $V_0\cap \ker p =\{0_B\}$. 
  Let $V$ be an open neighborhood of $b$ such that $V-V\subseteq V_0$. 
  Then $W := (U \times V) \cap \Sigma_{\phi}$ is an open neighborhood of $(\gamma, b)$ in $\Sigma_{\phi}$.
  Given $(\gamma_1, b_1), (\gamma_2, b_2) \in W$ such that $\pi_1(\gamma_1, b_1) = \pi_1(\gamma_2, b_2)$,
  then $\gamma_1 = \gamma_2$ and 
  \[
  p(b_1 - b_2) = p(b_1) - p(b_2) = \phi(\gamma_1) - \phi(\gamma_2) = 0.
  \]
  Hence 
  \[
  b_1 - b_2 \in (V - V) \cap \ker p \subset V_0\cap \ker p = \{0_B\}
  \]
  and so $b_1 = b_2$.
  Thus the restriction of $\pi_1$ to $W$ is injective and therefore a local homeomorphism.
  It follows easily that $\Sigma_{\phi}$ is \'etale.
\end{proof}

The following result generalizes an initial segment of the exact sequence given in 
\cite[Proposition 3.3]{kum:jot88}.

\begin{proposition}
  Given a short exact sequence  of locally compact  abelian groups
  \[
  0\rightarrow A\xrightarrow{i} B\xrightarrow{p} C\rightarrow 0,
  \]
  there is an exact sequence 
  \[
  0\rightarrow Z_\Gamma(A)\xrightarrow{i_{*}} Z_\Gamma(B)\xrightarrow{p_{*}}
  Z_\Gamma(C)\xrightarrow{\delta} T_\Gamma(A),
  \]
  where  $\delta(\phi):=[\Sigma_\phi]$ for  $\phi\in Z_\Gamma(C)$.
\end{proposition}
\begin{proof}
  Since $Z_\Gamma$ is left exact we only need to show that
  $\operatorname{Im}p_*=\operatorname{ker}\delta$. We have that $\phi\in \operatorname{ker}\delta$
  if and only if $\Sigma_\phi$ is properly isometric to the trivial twist. By Lemma
  \ref{lem:twist_prop}, this happens if and only if
  $\phi=p_*\tilde{\phi}$ for some $\tilde{\phi}\in Z_\Gamma(B)$.
\end{proof}

In the following example we consider the case  $A = \Z$, $B = \R$ and $C = \T$.
In a certain  case where the groupoid $\Gamma$ is equivalent to a space $X$, we indicate how 
the construction of the groupoid $\Sigma_\phi$ is related to the boundary map of \v{C}ech cohomology 
  $\check{H}^1(X, \TT)  \to \check{H}^2(X, \Z)$, where $\check{H}^1(X, \TT)$ is the first \v{C}ech cohomology 
of $X$ with coefficients in $\TT$, the sheaf of germs of continuous $\T$-valued functions, and $\check{H}^2(X, \Z)$ 
is the second \v{C}ech cohomology of $X$ with coefficients in the constant sheaf with fiber $\Z$
(see \cite[\S{4.1}]{rw:morita} for background on \v{C}ech cohomology).
This example is inspired by results in \cite{RaeRos88}.

\begin{example}\label{ex:cech}
Let $X$ be a Hausdorff second countable locally compact space; thus $X$ is also paracompact.  
Let $\mfU := \{ U_i \}_{i \in I}$ be a locally finite open cover of $X$. 
Set $U_{ij} := U_i \cap U_j$ for $i, j \in I$ and similarly 
$U_{ijk} := U_i \cap U_j \cap U_k$ for $i, j, k \in I$.
We now construct an \'etale groupoid associated to $\mfU$ 
(see \cite[\S{1}, Remark 3]{RaeTay85} and \cite[Example III.1.0]{Hae79}).  
Let $\Gamma_\mfU := \{ (x, i, j) : x \in U_{ij} \}$ and 
$\Gamma_\mfU^0 := \{ (x, i, i) : x \in U_{i} \}$;
note that $\Gamma_\mfU^0$ may be identified with the disjoint union $\bigsqcup_{i \in I} U_{i}$
and that $\Gamma_\mfU$ may be identified with the disjoint union $\bigsqcup_{i, j \in I} U_{ij}$.
It is routine to check that with the topology obtained from this identification and with 
the following structure maps $\Gamma_\mfU$ is an \'etale groupoid:
\begin{align*}
s(x, i, j) &= (x, j, j) \\
r(x, i, j) &= (x, i, i) \\
(x, i, j)^{-1} &= (x, j, i) \\
(x, i, j)(x, j, k) &= (x, i, k).
\end{align*}
Note that $\Gamma_\mfU$ may also be viewed as the natural groupoid associated to the  
local homeomorpism $\bigsqcup_{i \in I} U_i \to X$ (see \cite{Kum83}).
Next suppose that we are given a \v{C}ech 1-cocycle $\lambda = \{ \lambda_{ij} \}_{i, j \in I}$ 
with coefficients in $\TT$.  
So $\lambda_{ij} : U_{ij} \to \T$ is continuous for all $i, j \in I$ and
$\lambda_{ik}(x) = \lambda_{ij}(x)\lambda_{jk}(x)$ for all $i, j, k \in I$ and $x \in U_{ijk}$. 
It is routine to check that the map $\phi : \Gamma_\mfU \to \T$ given by
$\phi((x, i, j)) = \lambda_{ij}(x)$ is a continuous groupoid 1-cocycle.

Now suppose that each $\lambda_{ij}$ lifts, that is, for each $i, j \in I$ 
there is a continuous function $\tilde{\lambda}_{ij} : U_{ij} \to \R$ such 
that $\lambda_{ij} = e \circ \tilde{\lambda}_{ij}$ where $e(t) = e^{2{\pi}\sqrt{-1}t}$ for $t \in \R$.
We may assume that $\tilde{\lambda}_{ii} = 0$ for all $i \in I$.  
We observe that for $i, j, k \in I$ the formula
\[
(\lambda^\star)_{ijk}(x) := \tilde{\lambda}_{ij}(x) + \tilde{\lambda}_{jk}(x) - \tilde{\lambda}_{ik}(x)
\]
for  $x \in U_{ijk}$ defines a continuous integer-valued function.
A routine computation shows that for all $i, j, k, \ell \in I$ and $x \in U_{ijk\ell}$
\[
(\lambda^\star)_{jk\ell}(x) - (\lambda^\star)_{ik\ell}(x) + 
(\lambda^\star)_{ij\ell}(x) - (\lambda^\star)_{ijk}(x) = 0,
\]
that is, $(\lambda^\star)_{ijk}$ gives a \v{C}ech 2-cocycle with values in $\Z$ 
(i.e.\ the constant sheaf with fiber $\Z$).  
It is normalized in the sense that $(\lambda^\star)_{ijk} = 0$ if $j = i$ or $j = k$.
As in \cite{RaeTay85} we may construct a groupoid 2-cocycle $\phi^\star$ by the formula
\[
\phi^\star((x, i, j),(x, j, k)) = (\lambda^\star)_{ijk}(x) 
\]
for all $x \in U_{ijk}$.
We obtain a twist $\Sigma$ by $\Z$ over $\Gamma_\mfU$ determined by $\phi^\star$
(see \cite[Prop. I.1.14]{ren:groupoid}).  
Indeed, let $\Sigma := \{ (n, (x, i, j)) : n \in \Z, x \in U_{ij} \}$.  
We identify $\Gamma_\mfU^0$ with $\Sigma^0$ via the map $(x, i, i) \mapsto (0, (x, i, i))$.  
The range and source maps factor through those of $\Gamma_\mfU$. 
The remaining structure maps are given by
\begin{align*}
(m, (x, i, j))(n, (x, j, k)) &:= (m + n + (\lambda^\star)_{ijk}(x), (x, i, k)) \\
(n, (x, i, j))^{-1} &:= (-n - (\lambda^\star)_{jij}(x), (x, j, i)).
\end{align*}
We claim that $\Sigma \cong \Sigma _\phi$.  Define $\xi : \Sigma \to \Gamma_{\mfU}\times \R$ by
\[
\xi((n, (x, i, j))) = ((x, i, j), \tilde{\lambda}_{ij}(x) + n).
\]
We first note that $\xi$ induces an identification of the unit spaces. 
Given a pair of composable elements $(m, (x, i, j))$, $(n, (x, j, k)) \in \Sigma$ we have
\begin{align*}
\xi((m, (x, i, j))(n, (x, j, k))) &= \xi((m + n + (\lambda^\star)_{ijk}(x), (x, i, k))) \\
   &= ((x, i, k), \tilde{\lambda}_{ik}(x) + m + n + (\lambda^\star)_{ijk}(x)) \\
   &= ((x, i, k), m + n + \tilde{\lambda}_{ij}(x) + \tilde{\lambda}_{jk}(x)) \\
   &= ((x, i, j), \tilde{\lambda}_{ij}(x) + m) ((x,j, k), \tilde{\lambda}_{jk}(x) + n) \\
   &= \xi((m, (x, i, j)))\xi((n, (x, j, k))).
\end{align*}
Hence, $\xi$ is a groupoid homomorphism.  
Moreover $\xi$ is injective and its image coincides with $\Sigma_\phi$.  
It follows that the isomorphism class of $\Sigma$ does not depend on the specific choice 
of the $\tilde{\lambda}_{ij}$.  
\end{example}

\begin{remark}\label{iso}
With notation as in the above example any (normalized) \v{C}ech 2-cocycle with values in the 
constant sheaf  with fiber $\Z$ gives rise to a groupoid as above which is isomorphic to a pullback.
Let $\RR$ be the sheaf of germs of continuous $\R$-valued functions on $X$.
Since $\RR$ is a fine sheaf, we have $H^n(X, \RR) = 0$ for $n \ge 1$.  
Moreover since $0 \to \Z \to \RR \to \TT \to 0$ is exact, the connecting map of 
the long exact sequence of cohomology  $\partial^n : \check{H}^n(X, \TT) \to \check{H}^{n+1}(X, \Z)$
is an isomorphism for $n >0$ (see \cite[Theorem 4.37]{rw:morita}).   
In Example \ref{ex:cech}  we have $[\lambda^\star] = \partial^1([\lambda])$.
\end{remark}

\begin{remark}
Note that every \v{C}ech 2-cocycle with values in an arbitrary sheaf of 
abelian groups is cohomologous to a normalized cocycle $a = \{ a_{ijk} \}_{ijk}$ 
(that is, $a_{ijk} = 0$ if $j = i$ or $j = k$).
Indeed given a \v{C}ech 2-cocycle $c = \{ c_{ijk} \}_{ijk}$ a routine calculation shows that
$c_{iij}(x) = c_{iii}(x) = c_{jii}(x)$ for all $i, j \in I$ and $x \in U_{ij}$.
For $i, j \in I$, define a 1-cochain $\lambda$ by $\lambda_{ij} = 0$ if $i \ne j$ and $\lambda_{ii}(x) = c_{iii}(x)$ 
for $x \in U_{ii}$.  Then $a := c - d^1\lambda$ is a normalized 2-cocycle cohomologous to $c$.
(Recall that $(d^1\lambda)_{ijk}(x) := \lambda_{jk}(x) - \lambda_{ik}(x) + \lambda_{ij}(x)$ for $x \in U_{ijk}$). 
\end{remark}

\section{The structure of the $C^*$-algebra $C^*(\Sigma_\phi)$}
\label{sec:struct}

Let $\Gamma$ be a locally compact Hausdorff groupoid endowed with Haar system
$\{\hsg^u\}_{u\in\go}$ and let
 \[
  0\rightarrow A\xrightarrow{i} B\xrightarrow{p} C\rightarrow 0,
 \]
be a short exact sequence of locally compact abelian groups. In the sequel, we identify $A$ with
its image $i(A)$ in $B$.  Let $\phi\in Z_\Gamma(C)$ and let
$\Sigma_\phi$ be the obstruction twist as in \eqref{eq:pullback}. 
We prove below that  $C^*(\Sigma_\phi)\cong
\ind_{\widehat{C}}^{\widehat{B}}C^*(\Gamma)$ if $\Gamma$ is amenable (see Theorem \ref{thm:main}).
We begin by describing the action of $\widehat{C}$ on $C^*(\Gamma)$.
The following lemma follows immediately from Proposition II.5.1(iii) of \cite{ren:groupoid}.

\begin{lemma}\label{lem:action}
Assume that $C$ is a locally compact abelian group. Given $\phi\in
Z_\Gamma(C)$, the map
\[
\alpha^{\phi}_t(f)(\gamma)=\langle t,\phi(\gamma)\rangle f(\gamma),
\]
for $f\in C_c(\Gamma)$, $t\in\widehat{C}$, and $\gamma\in\Gamma$,
defines  a strongly continuous action
$\alpha^{\phi} :\widehat{C}\to \Aut C^*(\Gamma)$.  
\end{lemma}

We define next a Haar system on $\Sigma_\phi$. 
Choose Haar measures  $\mu_A$,
$\mu_B$, and $\mu_C$ on $A$, $B$, and, respectively $C$ such that:
\[
\int_B f(b)\,d\mu_B(b)=\int_C \int_A f(b+a)\,d\mu_A(a)d\mu_C(p(b)).
\]
One can always make such a choice since $A$, $B$, and $C$ are locally compact abelian
groups (see,  e.g. \cite[page 79]{DeiEct14}).
Let $\pi_1:\Sigma_\phi\to \Gamma$ be the projection map (see Lemma \ref{lem:twist_prop}).
Note that $\pi_1^{-1}(\gamma) = \{ \gamma \} \times A_\gamma$ where
$A_\gamma := \{  b\in B : \phi(\gamma)=p(b) \}$
is a coset in $B$, since $A_\gamma = b_{\gamma}+A$ for some $b_\gamma \in B$ with
$\phi(\gamma)=p(b_\gamma)$. 
We write $\mu_\gamma$ for the measure defined on $A_\gamma$ via
\[
\int_{A_\gamma} f(\gamma,b)\,d\mu_\gamma(b):=\int_Af(\gamma, b_\gamma+a)\,d\mu_A(a).
\]
The measure $\mu_\gamma$ is independent of the choice of $b_\gamma$
because $\mu_A$ is a Haar measure on $A$.
We define now a Haar system
$\{\hsi^u\}_{u\in\go}$ on $\Sigma_\phi$ via
\begin{equation}
  \label{eq:HaarSyst}
  \int_{\Sigma_\phi^u}f(\gamma,b)\,d\hsi^u(\gamma,b)=\int_{\Gamma^u}\int_{A_\gamma}
  f(\gamma,b)\,d\mu_\gamma(b)d\hsg^u(\gamma), 
\end{equation}
for all $u\in\go$. It is easy to check that \eqref{eq:HaarSyst} defines
 a Haar system on $\Sigma_\phi$ using the fact that
$\{\hsg^u\}$ is a Haar system on $\Gamma$ and $\mu_A$ is a Haar
measure on $A$.

The main goal of this section is to prove that $C^*(\Sigma_\phi)$
is isomorphic to the induced algebra $\ind_{\hat{C}}^{\hat{B}}
(C^*(\Gamma),\alpha^\phi)$ using Theorem \ref{thm:indalg}. 
We begin by defining a map $\rho:C_c(\Sigma_\phi)\to C_c(\Gamma)$ via
\begin{equation}
  \label{eq:rhomap}
  \rho(f)(\gamma)=\int_{A_\gamma} f(\gamma,b)\,d\mu_\gamma(b).
\end{equation}
\begin{lemma}\label{lem:rho}
  With notation as above, the map $\rho$ defined in \eqref{eq:rhomap} extends to  a surjective  
  $*$-homomorphism $\rho:C^*(\Sigma_\phi)\to C^*(\Gamma)$ which factors through the map 
  $\iota: C^*(\Sigma_\phi)\to M(C^*(\Gamma \times B))$ given in Subsection \ref{subsec:mult-algebr-group}. 
  Moreover, if either $\Gamma$ is amenable or $\Sigma_\phi$ is a clopen subset of $B \times\Gamma$,
  then $\iota$ is injective.
\end{lemma}
\begin{proof}
  We prove first that $\rho$ is a $*$-homomorphism. Let $f$ and $g$
  in $C_c(\Sigma_\phi)$ and let $\gamma\in 
  \Gamma$. Then, using the fact that $\mu_A$ is a Haar measure on $A$,
  we obtain 
  \begin{align*}
    \rho(f*g)(\gamma)&=\int_{A_\gamma}(f*g)(\gamma,b)\,d\mu_\gamma(b)\\
                     &=\int_{A_\gamma}\!
                       \int_{\Gamma^{r(\gamma)}}\!\int_{A_\alpha}f(\alpha,b^\prime)g(\alpha^{-1}\gamma,-b^\prime+b)\,d\mu_\alpha(b^\prime)d\hsg^{r(\gamma)}(\alpha)d\mu_\gamma(b)\\
                     &=\int_{\Gamma^{r(\gamma)}}\!\int_{A_\alpha}f(\alpha,b^\prime)\int_{A_\gamma}g(\alpha^{-1}\gamma,-b^\prime+b)\,d\mu_\gamma(b)d\mu_\alpha(b^\prime)d\hsg^{r(\gamma)}(\alpha)\\
                     &=\int_{\Gamma^{r(\gamma)}}\rho(f)(\alpha)\rho(g)(\alpha^{-1}\gamma)\,d\hsg^{r(\gamma)}(\alpha)\\
                     &=\bigl(\rho(f)*\rho(g)\bigr)(\gamma).
  \end{align*}
  It is easy to show that $\rho(f^*)=\rho(f)^*$ for all $f\in C_c(\Sigma_\phi)$.

  Recall from Subsection \ref{subsec:mult-algebr-group} (see
  \cite[Proposition II.2.4]{ren:groupoid})  that there is  a bounded
  $*$-homomorphism from $C^*(\Sigma_\phi)$ into $M(C^*(\Gamma\times B))$. Since
  $C^*(\Gamma\times B)$ is $*$-isomorphic to $C^*(\Gamma)\otimes
  C^*(B)$ and $C^*(B)$ is isomorphic to $C_0(\hat{B})$ it follows that
  $M(C^*(\Gamma\times B))$ is $*$-isomorphic to $C_b(\hat{B},M(C^*(\Gamma)))$ 
  where $M(C^*(\Gamma))$ is given the strict topology
  (see \cite[Corollary 3.4]{AkePedTom73}). 
  Then $\rho$ is just the composition
  of the above bounded $*$-homomorphisms with evaluation at $0$. Hence $\rho$
  extends to a bounded $*$-homomorphism $\rho:C^*(\Sigma_\phi)\to
  M(C^*(\Gamma))$. However, since $\rho(C_c(\Sigma_\phi))\subseteq
  C_c(\Gamma)$,  $C_c(\Sigma_\phi)$ is dense in $C^*(\Sigma_\phi)$,
  and $C_c(\Gamma)$ is dense in $C^*(\Gamma)$, it follows that $\rho$
  is a bounded $*$-homomorphism from $C^*(\Sigma_\phi)$ into
  $C^*(\Gamma)$.

  To prove  that $\rho$ is surjective let $\mfb$ be a
  \emph{Bruhat approximate cross-section} for $B$ over $A$. Recall
  (see, for example, Proposition C.1 of \cite{rw:morita}) that
  $\mfb:B\to[0,\infty)$ is a continuous function such that
  $\operatorname{supp}\mfb\cap (K+A)$ is compact for every compact set
  $K$ in $B$, and such that $\int_A\mfb(b+a)d\mu_A(a)=1$ for all $b\in
  B$. Given $g\in C_c(\Gamma)$ define $f\in C_c(\Sigma_\phi)$ via
  $f(\gamma,b)=g(\gamma)\mfb(b)$. It follows immediately that
  $\rho(f)=g$ and, hence, $\rho$ is surjective.
 
  The final assertion follows from Proposition \ref{prop:discretecase}
  if $\Sigma_\phi$ is a clopen subset of $\Gamma \times B$. 
  If $\Gamma$ is amenable the assertion follows from Proposition
  \ref{prop:amenablecase} and the fact that  $\Gamma \times B$ is
  amenable (see \cite[Prop. 5.1.2]{anaren:amenable00}). 
\end{proof}
We are now ready to state and prove the main result of the paper.
\begin{theorem}\label{thm:main}
  Let $\Gamma$ be a locally compact Hausdorff amenable groupoid endowed with Haar system
$\{\hsg^u\}_{u\in\go}$ and let  
  \[
  0\rightarrow A\xrightarrow{i} B\xrightarrow{p} C\rightarrow 0,
 \]
  be a short exact sequence of locally compact abelian groups. 
  Let $\phi\in Z_\Gamma(C)$ and let
  $\Sigma_\phi$ be the obstruction twist as in \eqref{eq:pullback} endowed with the
  Haar system defined in \eqref{eq:HaarSyst}. Then there are a
  surjective  $*$-homomorphism $\rho:C^*(\Sigma_\phi)\to C^*(\Gamma)$,  a strongly
  continuous action $\gamma:\hat{B}\to \Aut(C^*(\Sigma_\phi))$ 
  and a $\hat{B}$-equivariant homomorphism
  $i:C_0(\hat{B}/\hat{C})\to Z(M(C^*(\Sigma_\phi)))$ such that the four
  conditions of Theorem \ref{thm:indalg} are satisfied.
  Therefore there is a unique $\hat{B}$-equivariant isomorphism
  $\Psi :C^*(\Sigma_\phi)\to\ind_{\hat{C}}^{\hat{B}}C^*(\Gamma)$
  defined via
  \[
  \Psi(f)(t)=\gamma_{-t}(f)
  \] for $f\in C_c(\Sigma_\phi)$ and $t\in \hat{B}$.
\end{theorem}
\begin{proof}
  The existence of the map $\rho:C^*(\Sigma_\phi)\to
  C^*(\Gamma)$ is proved in Lemma \ref{lem:rho}. The strongly
  continuous action $\gamma:\hat{B}\to \Aut(C^*(\Sigma_\phi))$ is
  given by
    \begin{equation*}
    \gamma_t(f)(\sigma,b)=\langle t,b\rangle f(\sigma,b).
  \end{equation*}
  In order to define the map
  $i:C_0(\hat{B}/\hat{C})\to Z(M(C^*(\Sigma_\phi)))$ we identify
  $\hat{B}/\hat{C}$ with $\hat{A}$ and we 
  also identify $C_0(\hat{A})$ with $C^*(A)$ via the Fourier
  transform. Recall that we view the multiplier algebra
  $M(C^*(\Sigma_\phi))$  as the $C^*$-algebra of the adjointable
  operators on the Hilbert $C^*$-module
  $C^*(\Sigma_\phi)_{C^*(\Sigma_\phi)}$. Then we define $i:C^*(A) \to
  Z(M(C^*(\Sigma_\phi)))$ via
  \begin{equation}\label{eq:i}
    (i(h)f)(\sigma)=\int_A h(a)f((r(\sigma),-a)\sigma)\,d\mu_A(a)
  \end{equation}
  for all $\sigma\in \Sigma_\phi$,  $h\in C_c(A)$ and 
  $f\in C_c(\Sigma_\phi)$. A straightforward but tedious
  computation shows that $i(h)$ is an adjointable operator  on
  $C^*(\Sigma_\phi)_{C^*(\Sigma_\phi)}$: let $f$ and $g$ be in
  $C_c(A)$ and let $\sigma=(\gamma,b)\in \Sigma_\phi$. Then
  \begin{align*}
    \langle i(h)f\,,\,g\rangle(\gamma,b)&=(i(h)f)^**g(\gamma,b)\\
                                        &=\int_{\Sigma_\phi^{r(\gamma)}}(i(h)f)^*(\eta,b^\prime)g(\eta^{-1}\gamma,-b^{\prime}+b)\,d\nu^{r(\gamma)}(\eta,b^{\prime})\\
                                        &=\int_{\Gamma^{r(\gamma)}}\int_{A_\gamma}\overline{i(h)f(\eta^{-1},-b^\prime)}g(\eta^{-1}\gamma,-b^{\prime}+b)\,d\mu_{\gamma}(b^\prime)d\lambda^{r(\gamma)}(\eta)\\
    \intertext{which using \eqref{eq:i}, interchanging the integrals
    on $A$, using the fact that $\mu_A$ is a Haar measure on $A$, and
    interchanging back the integrals, equals}\\
                                        &\int_{\Gamma^{r(\gamma)}}\int_{A_\gamma}\overline{f(\eta^{-1},-b^\prime)}(i(h^*)g)(\eta^{-1}\gamma,-b^{\prime}+b)\,\,d\mu_{\gamma}(b^\prime)d\lambda^{r(\gamma)}(\eta)\\
    &=\langle f,i(h^*)(g)\rangle (\gamma,b).
  \end{align*}
  Hence $i(h)$ is an adjointable operator on
  $C^*(\Sigma_\phi)_{C^*(\Sigma_\phi)}$. Moreover, $i(h)$ belongs to
  $Z(M(C^*(\Sigma_\phi)))$ since $\Sigma_\phi$ is a twist. 
  Note that as in Remark \ref{rem:dual} $i$ is determined by the map $u : A \to UM(C^*(\Sigma_\phi))$ 
  given as the composition 
  \[
  A \to M(C^*(A) \otimes C_0(\Gamma^0)) \to M(C^*(\Sigma_\phi)).
  \]
  Then condition (i) of Theorem \ref{thm:indalg} follows from the fact
  that $\rho(u_{a}f) = \rho(f)$ for all $a \in A$ and  
  $f \in C_c(\Sigma_\phi)$.
  The nondegeneracy of $i$ (see Remark \ref{rem:nondeg}) follows from the
  fact that $u_0 =  1_{M(C^*(\Sigma_\phi))}$. 
  Hence condition (iv) of Theorem \ref{thm:indalg} holds.  Condition
  (ii) follows by a straightforward computation. 
  Since $B \times \Gamma$ is amenable, it follows by Lemma \ref{lem:rho} that 
  \[
  \iota: C^*(\Sigma_\phi)\to M(C^*(\Gamma \times B)) \cong C_b(\hat{B},M(C^*(\Gamma)))
  \] 
  is injective.
  Note that $\iota$ is $\hat{B}$-equivariant with respect to natural actions.  
  Hence, if $\rho(\gamma_\chi(f)) = 0$ for all $\chi \in \hat{B}$, we
  have $f  = 0$ and so condition (iii) holds. 
\end{proof}

\begin{remark}\label{rem:dual2}
As noted in  the proof of the above theorem (see Remark \ref{rem:dual}) one can replace conditions (i) and (iv) of  Theorem \ref{thm:indalg} by the following two conditions on the corresponding map
$u : A \to UM(C^*(\Sigma_\phi))$:
condition (i) may be replaced by the requirement that $\rho(u_{a}f) = \rho(f)$ for all $a \in A$ and $f \in C^*(\Sigma_\phi)$ and condition (iv) may be replaced by the requirement that 
$u_0 = 1_{M(C^*(\Sigma_\phi))}$.
\end{remark}

\begin{remark}\label{clopen}
  Theorem \ref{thm:main} also holds if one replaces the requirement that
$\Gamma$ be amenable by the requirement that  
$\Sigma_\phi$  be a clopen subset of $\Gamma \times B$ (see Lemma \ref{lem:rho}).
\end{remark}

\begin{example}\label{ex:cech2}
With notation as in Example~\ref{ex:cech}, we first observe that $C^*(\Gamma_\mfU)$ is a continuous 
trace algebra with $\Prim C^*(\Gamma_\mfU) = X$ and trivial
Dixmier-Douady invariant $\delta(C^*(\Gamma))=0 $, that is, 
it is strong Morita equivalent to $C_0(X)$ (see \cite{Kum83}).
Next we let $\alpha : \Z \to \Aut C^*(\Gamma_\mfU)$ be the action determined by the 
$\T$-valued  groupoid 1-cocycle $\phi$ defined by the \v{C}ech 1-cocyle $\lambda_{ij}$
(see \cite[Proposition II.5.1]{ren:groupoid}); 
so 
\[
(\alpha_n f)(x, i, j) := \lambda_{ij}(x)^nf(x, i, j)
\]
for all $f \in C_c(\Gamma_\mfU)$, $(x, i, j) \in \Gamma_\mfU$ and $n \in \Z$.
It is straightforward to see that $\alpha$ fixes every ideal in $C^*(\Gamma_\mfU)$.

For each $i \in I$ the ideal $J_i$ in  $C^*(\Gamma_\mfU)$ corresponding to $U_i \sub X$ 
may be identified with $C^*((\Gamma_\mfU)_{V_i})$ where 
$V_i := \{ (x, j, j) : x \in U_{ij} \} \sub \Gamma_\mfU^0$ (hence $V_i
\cong \bigsqcup_{j \in I} U_{ij}$).
Note that $(\Gamma_\mfU)_{V_i} = \{ (x, j, k) : x \in U_{ijk} \}$; let 
$w_i : (\Gamma_\mfU)_{V_i}  \to \mathbb{R}$ be defined by
\[
w_i(x, j, k) := 
\begin{cases}
\lambda_{ji}(x) & \text{if } j = k, \\
0 & \text{ otherwise.}
\end{cases}
\]
Then $w_i$ may be identified with a multiplicative unitary for the ideal $C^*((\Gamma_\mfU)_{V_i})$
and $\alpha_1 |_{J_i} = \Ad w_i$.  Hence, $\alpha$ is locally unitary.
A quick calculation shows that $\lambda_{ij}w_i(x, k, k) = w_j(x, k, k)$ for all $x \in U_{ijk}$.
Therefore the Phillips-Raeburn obstruction $\eta(\alpha)$ is the class 
$[\lambda] \in \check{H}^1(X, \TT)$, the first \v{C}ech cohomology of $X$ with coefficients in $\TT$, 
the sheaf of continuous $\T$-valued functions on $X$ (see \cite[{\S}2.10]{PhiRae80}).
Using the usual isomorphism $\check{H}^1(X, \TT) \cong \check{H}^2(X, \Z)$ we identify $\eta(\alpha)$
with $[\lambda^\star] \in \check{H}^2(X, \Z)$.  
Note that the element $\eta(\alpha)$ may also be identified with the class of 
$\Prim C^*(\Gamma_\mfU) \rtimes_\alpha \Z$ regarded as a circle bundle 
under the dual action of $\T$ (see \cite[p.~224]{PhiRae84}).

By Theorem \ref{thm:main} we have 
\[
  C^*(\Sigma_\phi) \cong \ind_{\Z}^{\R}(C^*(\Gamma), \alpha).
\]
Next we observe that $C^*(\Sigma_\phi)$ is a continuous trace algebra (see the remark preceding
\cite[Lemma 3.3]{RaeRos88}) and that $\Prim C^*(\Sigma_\phi) \cong \T \times X$ 
(see the proof of \cite[Proposition 3.4]{RaeRos88}).
Since $\delta(C^*(\Gamma)) = 0$, it follows by  \cite[Corollary 3.5]{RaeRos88} that 
\[
\delta(C^*(\Sigma_\phi)) = z \times \eta(\alpha) = z \times [\lambda^\star] 
\in \check{H}^3(\T \times X, \Z)
\] 
where $z$ is the standard generator of $\check{H}^1(\T, \Z)$.
\end{example}

\printbibliography
\end{document}